\documentclass[12pt,a4paper]{article}
\usepackage[active]{srcltx}
\usepackage{amsfonts}
\usepackage{amsmath}
\usepackage{amsbsy}
\usepackage{amsxtra}
\usepackage{latexsym}
\usepackage{amssymb}
\usepackage{url}
\usepackage{cite}
\usepackage{pstricks,pst-node}
\makeatletter
\g@addto@macro\thesection.
\makeatother

\newtheorem{theorem}{Theorem}
\newtheorem{lemma}{Lemma}
\newtheorem{corollary}{Corollary}
\newtheorem{proposition}{Proposition}
\newtheorem{remark}{Remark}
\newtheorem{definition}{Definition}
\newtheorem{example}{Example}

\DeclareMathOperator{\cone}{cone}
\DeclareMathOperator{\intr}{int}

\DeclareMathOperator{\bdr}{bdr}
\DeclareMathOperator{\spa}{sp}

\DeclareMathOperator{\ddm}{dim}

\DeclareMathOperator{\argmin}{argmin}

\DeclareMathOperator{\core}{core}

\DeclareMathOperator{\rng}{rng}

\DeclareMathOperator{\ray}{ray}
\DeclareMathOperator{\relintr}{relint}

\newcommand{\lang}{\langle}
\newcommand{\rang}{\rangle}

\newcommand{\R}{\mathbb R}

\newenvironment{proof}{{\noindent\bf Proof.}}{\hfill$\Box$\\}

\begin{document}

\title{Mutually polar retractions on convex cones}
\author{A. B. N\'emeth\\Faculty of Mathematics and Computer Science\\Babe\c s 
Bolyai University, Str. Kog\u alniceanu nr. 1-3\\RO-400084
Cluj-Napoca, Romania\\email: nemab@math.ubbcluj.ro}
\date{}
\maketitle

\begin{abstract} 

Two retractions $Q$  and $R$ on closed convex cones $M$ and
respectively $N$ of a Banach space
are called mutually polar if $Q+R=I$ and $QR=RQ=0.$
This note investigates the
existence of a pair of mutually polar retractions
for given cones $M$ and $N$.
It is shown that if $\ddm N=1$ (or $\ddm M =1$) then the
retractions are subadditive with respect to the order
relation their cone ranges endow.

\end{abstract}

\section{Introduction}

If $X$ is a nonempty set, then the mapping $Q:X \to X$ is said \emph{idempotent}, if $Q^2=Q$.
The aim of this note is to investigate idempotent mappings with real 
Banach space codomains and convex cone invariant sets.

\begin{definition}\label{retract}
Let $X$ be a Banach space. The mapping $T:\,X\to X$ is called
\emph{retraction} if:
\begin{enumerate}
\item [(i)] It is a continuous idempotent mapping;
\item [(ii)] It is \emph{positive
homogeneous}, that is, $T(tx)=tTx$ for every $x\in X$ and every $t\in \R_+=[0,+\infty)$;
\item [(iii)] $T(X)=\rng T$ is a non-empty, non-zero closed convex cone.
\item [(iv)] $Tx\in \bdr \rng T$ for any $x\in X\setminus \rng T.$
\end{enumerate}
\end{definition} 

\begin{definition}\label{muture}
Let $X$ be a Banach space, $0$ its zero mapping and $I$ its identity mapping. 
The mappings $Q,\,R:\,X\to X$ are called
\emph{mutually polar retractions} if 
\begin{enumerate}
\item [(i)] $Q$ and $R$ are retractions, 
\item[(ii)] $Q+R=I$,
\item[(iii)] $QR=RQ=0$.
\end{enumerate}
\end{definition}

 We will say that $Q$ and $R$ 
are \emph{polar of each other}.

(For the detailed terminology and examples see the next section.)

\begin{remark}\label{proper}

{\bf In our note \cite{NemethNemeth2020} the  retraction $Q$ onto a cone
is said \emph{proper} if the pair $Q$ and $I-Q$ are mutually
polar retractions.} 

For technical reasons, instead of a proper retraction, we will use the notion 
of a pair of mutually polar retractions. {\bf We use tacitly this equivalence}
in our proofs, by applying the results of the above cited note .

\end{remark}

While searching conditions for the existence of
asymmetric vector norms, in the recent note \cite{NemethNemeth2020}, we 
considered mutually polar retractions, related to the order
relations their cone ranges endow. 

The order theoretic restrictions seem
to be severe and restrict substantially the existence of mutually polar 
retractions. We will see that regardless of the order theoretic restrictions, 
the class of mutually polar retractions is substantially larger.

The main result of this note is to construct mutually polar
retractions with given cone ranges for a reasonable general class
of pairs of cones, and to relate our construction to order theoretic
restrictions on retractions. The notion of the \emph{transversal
mutually polar retractions} is introduced. If one of the cone ranges of a 
mutually polar pair of retractions is one dimensional, then they are 
transversal and both of them are subadditive
with respect to the order relation their cone ranges endow. 

In Section 2 we will introduce our terminology and will give important
examples of mutually polar retractions from the literature.

In Sections 3-5 the transversal mutually polar retractions are defined
and some of their important properties are proved.

Section 6 is concerned with some order theoretic properties of the retractions. 

Subsection 6.1 contains the proof of one of the main results of our note.
The proof is strongly related to the ideas and techniques in 
\cite{NemethNemeth2020}, some of which are reproduced there.

To keep the spirit of the note, we complete our exposition with
some of the main results in \cite{NemethNemeth2020}, which are reproduced 
without proofs in Subsections 6.2 and 6.3.


\section{Terminology and preliminaries }

Let $X$ be a real vector space.

The nonempty set $K\subset X$ is called a \emph{convex cone} if 
\begin{enumerate}
	\item[(i)] $\lambda x\in K$, for all $x\in K$ and $\lambda \in \R_+$ and if 
	\item[(ii)] $x+y\in K$, for all $x,y\in K$. 
\end{enumerate}

A convex cone $K$ is called \emph{pointed} if $K\cap(-K)=\{0\}$. 

A convex cone is called \emph{generating} if $K-K=X$.
 
If $X$ is a Banach space, then
a closed, pointed, convex cone with non-empty interior is called \emph{proper}.

The \emph{relation $\leq$ defined by the pointed convex cone $K$} is given by $x\leq y$ if and only if 
$y-x\in K$. Particularly, we have $K=\{x\in X:0\leq x\}$. The relation $\leq$ is an \emph{order relation}, that 
is, it is reflexive, transitive and antisymmetric; it is \emph{translation invariant}, that is, $x+z\leq y+z$, 
$\forall x,y,z\in X$ with $x\leq y$; and it is \emph{scale invariant}, that is, $\lambda x\leq\lambda y$, 
$\forall x,y\in X$ with $x\leq y$ and $\lambda \in \R_+$.

Conversely, for every $\leq$ scale invariant, translation invariant and antisymmetric order relation in $X$, there is a pointed 
convex cone $K$, defined by $K=\{x\in X:0\leq x\}$, such that $x\leq y$ if and only if $y-x\in K$. The cone $K$ 
is called the \emph{positive cone of $X$} and $(X,\leq)$ (or $(X,K)$) is called an \emph{ordered vector space}; we use
also the notation $\leq =\leq_K.$ 

The mapping $R:(X,\leq)\to (X,\leq)$ is said to be \emph{isotone} if $x\leq y\,\Rightarrow Rx\leq Ry$, and \emph{subadditive} if $R(x+y)\leq Rx+Ry$, for any $x,\,y\in X$.

The ordered vector space $(X,\leq)$ is called \emph{lattice ordered} if for every  $x,y\in X$, there exists 
$x\vee y:=\sup\{x,y\}$. In this case the positive cone $K$ is called a \emph{lattice cone}.
A lattice ordered vector space is called a \emph{Riesz space}. 
Denote $x^+=0\vee x$ and $x^-=0\vee (-x)$. Then, $x=x^+-x^-$, $x^+$ is 
called the \emph{positive part} of $x$ and $x^-$ is called the \emph{negative part} of $x$. 
The \emph{absolute value} of $x$ is defined by $|x|=x^++x^-$. The mapping 
$x\mapsto x^+$ is called the \emph{positive part mapping}.


\begin{example}\label{latt}

Let $(X,\leq)$ be an ordered Banach space with $K$ its closed positive cone.
The following assertions are part of the classical vector
lattice theory (see e.g. \cite{LuxemburgZaanen1971} or \cite{Schaefer1974}.)
 
\begin{enumerate}

\item Using the properties of the positive part operator we can see that
$^+$ is a proper retraction.

\item The operator $^+: X \rightarrow K$ is obviously idempotent.

\item We have $(tx)^+=tx^+\,\forall t\in \R_+,\, \forall x\in X$. 

\item Let $x^-= (-x)^+$. Then $x=x^+-x^-$, that is, $^+-^-=I$.

\item $(I-^+)^+=0$.

\item The mapping $-^- = I-^+$ is a proper retraction onto $-K$ .

\end{enumerate}

{\bf Hence, $^+$ and $-^-$ are mutually polar retractions with
the cone range $K$ and $-K$ respectively.}

Beside the above properties, the positive part operator $^+$

6. is \emph{isotone}, i.e., $x\leq_K y \Rightarrow x^+\leq_K y^+$;

7. is \emph{subadditive}, i.e., $(x+y)^+\leq_K x^++y^+,\;\forall \,x,\,y \in X$
(from the definition of the supremum).

\end{example}

In the particular case of $X=H$, where $(H,\lang,\rang)$ is a separable Hilbert
space of scalar product $\lang,\rang$, we will need some further notions.
Let $K\subseteq H$ be a closed convex cone. Recall that 
$$K^\circ=\{x\in H:\langle x,y\rangle\leq 0,\,\forall y\in K\}$$
is called the \emph{polar cone} of $K$. The cone $K^\circ$ is closed and convex, and if $K$ is generating, then $K^\circ$ is 
pointed (this is the case for example if $K$ is latticial).

\begin{example}\label{proj}

Let $P:H\to K$ be the \emph{projection mapping} onto the closed convex cone 
$K$, that is, the mapping defined by  
$$Px=\argmin \{\|x-y\|:y\in K\}.$$ 

The projection mapping $P$ is a proper retraction. This is the consequence of
the following theorem proved in \cite{Moreau1962}.

\begin{theorem} [Moreau]
	Let $H$ be a Hilbert space, $K,L\subset H$ two mutually polar closed 
	convex cones in $H$. Then, the following statements are equivalent:
	\begin{enumerate}
		\item[(i)] $z=x+y,~x\in K,~y\in L$ and $\lang x,y\rang=0$,
		\item[(ii)] $x=P_K(z)$ and $y=P_L(z)$.
	\end{enumerate}
\end{theorem}

Let us denote by $P$ the projection onto $K$. Since $K$ and $K^\circ$ are
mutually polar (Farkas lemma), we have from Moreau's theorem the identity

\begin{equation}\label{mor}
x=Px+(I-P)x \quad \textrm{with} \quad \lang Px, (I-P)x \rang =0,
\end{equation} 
and the important consequence, that if $x=u+v$ with $u\in K$, $v\in K^\circ$ 
and $\lang u,v\rang =0$, then we must have $u=Px$
and $v=(I-P)x.$

From (\ref{mor}) $Px=0$ if and only if $x\in (I-P)H=K^\circ$ and thus

\begin{equation}\label{ppro}
P(I-P)=0,\;\; (I-P)P=0.
\end{equation}

{\bf Thus $P$ is a proper retraction and hence $P$ and $I-P$ are mutually
polar retractions with the cone ranges $K$ and respectively $K^\circ$.}

\end {example}

When $(H,\lang,\rang)=(\R^m,\lang.\rang)$, the $m$-dimensional Euclidean space,
the set
$$ K=\{t^1x_1+\dots+t^m x_m:\;t^i\in \R_+,\;i=1,\dots,m\}$$
with $x_1,\,\dots,\,x_m$ linearly independent vectors is called a \emph{simplicial cone}.
A simplicial cone is closed, pointed and generating.

The simplicial cones are related to vector lattices through
the following important result of Youdine (\cite{Youdine1939}).
 
\begin{theorem}[Youdine]\label{Yu}
The pair $(\R^m,K)$ is a vector lattice with continuous lattice operations 
if and only if $K$ is a simplicial cone. 
\end{theorem}

(For this reason in the vector lattice theory simplicial
cones sometimes are called Youdine cones as well.)

\begin{definition}\label{HVL}

The Hilbert space $H$ ordered by the order relation induced by the cone $K$
is called a \emph{Hilbert vector lattice} if (i) $K$ is a lattice cone, 
(ii) $\||x|\|=||x\|,\,\forall x\in H$, (iii) $0\leq x\leq y$ implies $\|x\|\leq\|y\|.$

\end{definition}

The cone $K$ is called \emph{self-dual}, if $K=-K^\circ.$ If $K$
is self-dual, then it is a generating, pointed, closed convex cone.

In all that follows we will suppose that $\R^m$ is endowed with a
Cartesian reference system with the standard unit vectors $e_1,\dots,e_m$.

The set
\[\R^m_+=\{x=(x^1,\dots,x^m)\in \R^m:\; x^i\geq 0,\;i=1,\dots,m\}\]
is called the \emph{nonnegative orthant} of the above introduced Cartesian
reference system. A direct verification shows that $\R^m_+$ is a
self-dual cone. 

%

Beside the terminology introduced in this section, we need some 
standard notions (hyperplane, core, relative
interior etc.) and some classical results on cones
and convex sets from functional analysis and convex geometry
contained in the monographs \cite{KreinRutman1948},\cite{Rockafellar1970}.

\section{Transversal cones; transversal mutually polar\\ retractions}

\begin{definition}\label{trans}
The pair of closed, pointed cones $M$ and $N$ in the Banach space $X$ 
will be  called \emph{transversal}, if the following conditions hold:

\begin{enumerate}
\item $M\cap N=\{0\}$;
\item There exist a straight line $\Delta$, called \emph{transversal line} trough $0$ in $X$ such that
\begin{enumerate}
\item$(M\setminus \{0\})\cap \Delta \not=\emptyset$ and
$(N\setminus \{0\})\cap \Delta \not=\emptyset$;
\item $(\intr M\cup \intr N)\cap \Delta \not= \emptyset$.
\end{enumerate}
\end{enumerate}
\end{definition}

\begin{definition}\label{mutu}
The mutually polar retractions $Q,\,R:X\to X$ are called \emph{transversal}
with the transversal line $\Delta$ through $\{0\}$, if for each two dimensional
plane $L$ containing $\Delta$ it holds $Q:L\to L$
and $R:L\to L$.
\end{definition}

\begin{proposition}\label{eelso}

If $Q,\,R:X\to X$  are mutually polar retractions and one of their
cone ranges has interior points, then the cone ranges are transversal.

\end{proposition}

\begin{proof}

Suppose that $\intr \rng Q\not= \emptyset$. Take $z\in -\intr \rng Q$.
Then there exist $e\in \rng Q$ and $u\in \rng R$ such that $z=e+u.$
Then
$$u=z-e\in ( -\intr \rng Q - \rng Q)\cap \rng R \subset -\intr \rng Q\cap \rng R.$$

Hence $\Delta =\spa \{u\}$ is a transversal line of $\rng Q$ and $\rng R$.
\end{proof}

\begin{proposition}\label{ell}

Suppose that $Q, \,R:\,X\to X$ are mutually polar retractions
with $N=\rng R$. If $\dim N=1$, then they are transversal mutually
polar retractions with the
transversal line $\spa N$.
 
\end{proposition}

\begin{proof}

Let us take $x\in X \setminus(M\cup N)$. 
Denote $Y= \spa \{x,\, \spa N\}$. We should check that
$Q,\,R: Y\to Y$.

Take $x+y$ with some $y\in \spa N$. Obviously $R(x+y)\in N \subset Y$.
By definition $x+y =R(x+y) + Q(x+y).$ Hence, 
$$Q(x+y)=x+y-R(x+y) \in x+\spa N \subset Y.$$

\end{proof}

\section{Mutually polar retractions in $\R^2$}

Consider $\R^2$ endowed with a Cartesian reference system.

For a pair of transversal cones $K_1$ and $K_3$ in $\R^2$, the following two
geometrical configurations are possible:

\begin{enumerate}
	\item[(i)] 
	
Let $I,\, II,\, III,\, IV$ be the quadrants of this
reference system. Let $e_1,\, e_2,\,u_1,\,u_2\in \R^2$ such that
$e_1\in \intr I,\;e_2\in \intr II,\; u_1\in \intr III,\;u_2\in \intr IV.$
$$ K_1=\cone \{e_1,e_2\},\; \textrm{and}\;K_3=\cone \{u_1,u_2\}.$$

  \item[(ii)] 

With the above notations, let $e_1,\, e_2,\,u_1,\,u_2\in \R^2$ such that
$e_1\in \intr I,\;e_2\in \intr II,\; u=u_1=u_2\in III\cap IV\setminus \{0\}$.

$$ K_1=\cone \{e_1,e_2\},\; \textrm{and}\;K_3=\cone \{u\}.$$

\end{enumerate}

\begin{theorem}\label{ketdim}

If $M$ and $N$ are transversal closed pointed cones in $\R^2$, then they determine a
unique pair of mutually polar retractions $Q$ and $R$, such that the cone range of
$Q$ and $R$ are $M$ and $N$, respectively.

\end{theorem}

\begin{proof}

(a) Using the notations introduced above consider the following cones:
\begin{enumerate}

\item $ K_1= \cone \{e_1,\,e_2\}$,

\item $ K_2 =\cone \{e_2,\,u_1\}$,

\item $ K_3= \cone \{u_1,\,u_2\}$,

\item $ K_4 =\cone \{u_2,\,e_1\}$.

\end{enumerate}

Observe that the transversal cones $M=K_1$ and $N=K_3$ admit such a 
representation (accepting the particular case $u_1=u_2$).

Then $\intr K_i\cap \intr K_j=\emptyset$ if $i\not=j$ and
$\R^2=K_1\cup K_2 \cup K_3 \cup K_4.$

Let $x\in \R^2$ be arbitrary. We define the following operators
$Q, \,R: \R^2 \to \R^2$:

\begin{displaymath}
Qx=\left\{\begin{array}{ll}
x&\textrm{if $x\in K_{1}$}\\
\lambda e_1&\textrm{if $x\in K_{4}$\,and $x=\lambda e_1+\mu u_2$}\\
\lambda e_2&\textrm{if $x\in K_{2}$\,and $x=\lambda e_2+\mu u_1$}\\
0&\textrm{if $x\in K_{3}$}
\end{array} \right.
\end{displaymath}

\begin{displaymath}
Rx=\left\{\begin{array}{ll}
x&\textrm{if $x\in K_{3}$}\\
\mu u_2&\textrm{if $x\in K_{4}$\,and $x=\lambda e_1+\mu u_2$}\\
\mu u_1&\textrm{if $x\in K_{2}$\,and $x=\lambda e_2+\mu u_1$}\\
0&\textrm{if $x\in K_{1}$}
\end{array} \right.
\end{displaymath}

A direct verification shows that $Q$ is a proper retraction with the cone range
$K_1$, $R$ is a proper retraction with cone range
$K_3$, that $Q+R=I$, and $QR=RQ=0$. Hence, $Q$ and $R$ are mutually polar retractions.

(b) Our constructions  of the mutually polar proper retractions $Q$ and $R$ holds
also when $u_1=u_2 \in K_3=\{(0,-t):\,t\in \R_+\}\setminus \{0\}.$
In booth cases $K_1$ and $K_3$ are transversal cones.

(c) Suppose that $Q$ and $R$ are mutually polar retractions in $\R^2$.
We will show that their cone ranges are transversal.

From the relation $Q+R=I$, it follows that
\begin{equation}\label{2trr} 
K_1+K_3= \R^2,\;\, \textrm{with}\; K_1=\rng Q,\;K_3=\rng R.
\end{equation}

 This relation shows that $K_1$ and
$K_3$ cannot be both of dimension one. Hence, we can assume for instance that
$\intr K_1 \not=\emptyset.$ Then, by Proposition \ref{eelso}, the
cones are transversal.

(d) For technical reasons, for $e\in \R^2 \setminus \{0\}$, we denote by
$$\ray e =\{te:\,t\in \R_+\}$$
the one dimensional cone engendered by $e$.

Suppose that $e_1,\, e_2$ and $u_1,\,u_2$ are vectors generating $K_1=\rng Q$
and $K_3=\rng R$, respectively. 

By an appropriate choice of the the reference system, we can realize to the 
transversal line $\Delta$ of $K_1$ and $K_3$ be the perpendicular axis of the reference system. 

 Hence we must have such a geometric picture as
in (a) (or (b)). 

(d) We will finally see that $Q$ and $R$ are exactly of the form as in (a) 
(or (b)).

Take an element in the interior of $K_4 =\cone \{e_1,u_2\}$ , say $x=e_1+u_2$.
We will see that $Qx=e_1$ and $Rx=u_2$.

Obviously $x=e_1+u_2= Q(e_1+u_2)+R(e_1+u_2)\in K_1+K_3$, since $Q+R=I$. 

By definition we must have
$Qx\in \bdr K_1=\ray e_1\cup \ray e_2$. $Qx\not= 0$, since $Qx=0$ would lead to the contradiction  $x\in K_3$.
 
If we had $z=Qx\in \ray e_2,$ then it would follow, by the continuity
of $Q$, that the line segment $xw$, joining $x$ and $w\in \ray e_1$, is
mapped by $Q$ in a continuous curve on $\bdr K_1$ joining $w$ and $z$. Such a 
curve must contain $0$, as
the image by $Q$ of an interior point $y$ of the segment $xw$. The equality 
$Qy=0$ would imply that an interior point $y$
of $K_4$ is in $K_3$, which is a contradiction.

Hence, we must have $Qx\in \ray e_1.$ 

By a similar argument, $Rx\in \ray u_2$. Since $x$ can uniquely be represented 
in the positive quadrant of the reference system endowed by the vectors $e_1$
and $u_2$, we must have $Qx=e_1,\;Rx=u_2$. Exploiting the positive
homogeneity of retractions, it follows for an arbitrary $y\in K_4$, $y=\lambda e_1+\mu u_2,\;
\lambda,\,\mu \in \R_+$ that $Qy=\lambda e_1,\;Ry=\mu u_2$. 
(If $\ray u_1=\ray u_2,$ then $\spa u_i\cap \intr \cone \{e_1,\,e_2\} \not= \emptyset$
and our proof is similar.)

 This proves that $Q$ and $R$
are exactly the retractions constructed in (a) or (b).

\end{proof}
 
\begin{corollary}\label{ketdimc}

Suppose that $\R^2$ is equiped with a norm. In 
what follows we will suppose that $\|e_1\|=\|e_2\|=\|u_1\|=\|u_2\|=1$. 

From the formulas defining the mappings $Q$ and $R$, it follows that they
are continuous functions of the vectors $e_1,\,e_2,\,u_1,\,u_2.$

\end{corollary}

\section{Mutually polar retractions on pair of transversal cones}

\begin{proposition}\label{ktrr}
Suppose that $Q, \,R:\,X\to X$ are mutually polar retractions
with $M=\rng Q$, $N=\rng R$.  If $\dim N=1$, and $\intr M\not= \emptyset$,
then they are transversal cones with $\Delta = \spa N$ their transversal line.
\end{proposition}

\begin{proof}
We have from Proposition \ref{ell} that $\Delta$ is the transversal
line of the retractions $Q$ and $R$.
From the definition of transversal mutually polar
retractions and Theorem \ref{ketdim}, for each two-dimensional
plane $Y$ through $\Delta$ we must have that $M_Y=Y\cap M$ and $N_Y=Y\cup N$
are mutually polar two dimensional cones in $Y$. Since $\dim N_Y=1$, the line
$\spa N_Y$ (which in fact is exactly $\Delta$) must meet $M_Y$ in its
relative interior point $y$. From convex
geometric reasons, this means that $y\in \core M= \intr M $ \cite{Rockafellar1970}.
This proves that $M$ and $N$ are transversal cones with $\spa N$ their
transversal line.
\end{proof}

\begin{theorem}\label{foo}
Consider the transversal cones $M$ and $N$ in the Banach space $X$ with the transversal line $\Delta$.
If one of the following conditions hold
\begin{enumerate}
\item [(i)] $\intr M\not=\emptyset$, $\intr N\not= \emptyset$, or
\item [(ii)] $\dim N=1$,
\end{enumerate}
then there exists a unique pair of transversal mutually polar retractions $Q$ and $R$ with
transversal line $\Delta$ and the cone ranges $M$ and $N$, respectively.
\end{theorem}

\begin{proof}

(i) Suppose that $M$ and $N$ is a pair of transversal cones in $X$
with the transversal line $\Delta$.

We will construct a pair of mutually polar proper retractions $Q$ and $R$
with the cone range $M$ and $N$ respectively.

Take $x\in X\setminus (M\cup N)$ Let $H$ be the two dimensional plane
$$H= \spa\{x,\,\Delta\}.$$ 

The plane $H$ intersects the cones $M$ and $N$ in the pair of transversal cones
$M_H=M\cap H$ and $N_H=N\cap H$.

If we identify $H$ with $\R^2$ and take an appropriate reference system therein, 
then we can  suppose that

$$M_H=\cone\{e_1,\,e_2\},\;\; N_H= \cone \{u_1,\,u_2\},$$

with appropriate unit vectors $e_1$, $e_2$, $u_1$ and $u_2$ in $H$. Next we can construct the
mutually polar proper retractions $Q_H$ and $R_H$ in $H$ as it was done in
the proof of Theorem \ref{ketdim}.

Since $\intr M\not= \emptyset$ and $\intr N\not= \emptyset,$ from convex 
geometric reasons, the vectors $e_i,\, u_j, \|e_i\|=\|u_j\|=1,\,i,j=1,2$ depend continuously
on $x$ and the same thing is valid for $Q_H$ and $R_H$.

Define now the mappings

$$Qx=Q_Hx\; \textrm{if}\; x\in H,\;\; Rx=R_Hx ,\;\textrm{if}\;x\in H.$$

The pair of mappings $Q$ and $R$ is obviously a pair of
mutually polar proper retractions with the cone ranges
$M$ and $N$, respectively.

From Theorem \ref{ketdim} it follows that $Q_H$ and $R_H$ are well
determined by their cone ranges and hence so are $Q$ and $R$.

(ii)
Our constructions  of the mutually polar retractions $Q$ and $R$ hold
also when $u_1=u_2 \in K_3=\{(0,-t):\,t\in \R_+\}\setminus \{0\}.$

\end{proof}

\begin{corollary}\label{foo1}
If $Q,,\;R: X\to X$ are mutually polar retractions in the Banach
space $X$ with the cone ranges $M$ and respectively $N$ with
$\intr M \not=\emptyset$ and $\dim N=1$, then $Q$ and $R$ are well
defined by their cone ranges.
\end{corollary}

\begin{proof}

By Proposition \ref{ell} $Q$ and $R$ are transversal retractions,
and hence, from Theorem \ref{foo}, it follows that they are well defined by
their cone ranges.

\end{proof}


\section{Mutually polar retractions with order theoretic properties}

We will say that a retraction $S$ is subadditive or isotone, if
it has this property with respect to the order relation defined
by its cone range $\rng S$. Thus, by saying that the mutually polar
retractions $Q$ and $R$ are subadditive or isotone, we mean that $Q$ has the 
corresponding property with respect the order relation endowed by $\rng Q$ and 
$R$ has the corresponding property with respect to the order relation endowed by $\rng R$.

\subsection{Subadditive one range retractions}

For the sake of completeness of the proofs, in this subsection we 
will reproduce some results and proofs from Section 5 and Section 6  of
\cite{NemethNemeth2020}.

In \cite{NemethNemeth2020} a retraction $R:\,X\to X$ was called \emph{one range},
if $N=\rng R=R(X)$ is  one dimensional. 

\begin{lemma}\label{oner}

If $R:\, X\to X$ is a one range retraction, then it is of form
$$Rx=q(x)u,\; \textrm{where}\; u\in N=\rng R,\;\textrm{and}\; q\; \textrm{is positive homogeneous functional with}\;q(u)=1.$$

\end{lemma}

\begin{proof}
The representation of $Rx$ in the form $q(x)u$, with some 
$u\in N\setminus \{0\}$ and $q$ a non-negative positive homogeneous 
functional is obvious.

By
definition, $R$ is idempotent, hence we must have
$$ q(x)u = Rx = R(Rx) = R(q(x)u)=q(x)Ru=q(x)q(u)u, $$
whereby $q(u)=1$.
\end{proof}

\begin{lemma}\label{onecomp}
For the one range retraction $T$ , with $Tx=q(x)u$, there exists the
idempotent mapping $S:\,X\to X$ such that $T+S=I$ and
$TS=ST=0$ if and only if
\begin{equation}\label{onec}
q(x-q(x)u)=0,\;\forall \;x\in X.
\end{equation}
\end{lemma}

\begin{proof}
Suppose that $S$ is an idempotent mapping with the property in the lemma.
Then $S=I-T$ and
\begin{equation}\label{onec1}
  Sx=S(Sx)=(I-T)((I-T)x)=(I-T)(x-Tx)=(I-T)(x-q(x)u)=x-q(x)u-T(x-q(x)u)=
(I-T)x-q(x-q(x)u)u=Sx-q(x-q(x)u)u,
\end{equation}
whereby we have relation (\ref{onec}).

Conversely, if (\ref{onec}) holds, then from (\ref{onec1}),  $S$ is idempotent.

Further,
$$ TSx=T(I-T)x=T(x-q(x)u)=q(x-q(x)u)u,$$
and by (\ref{onec}) it follows that $TS=0.$

From the idempotent property of $T$ it follows that
$$ST=(I-T)T=T-T^2=T-T=0.$$
\end{proof} 

\begin{remark}\label{megj}
\begin{enumerate}
\item We observe that for three non-linear mappings $E,\;F$ and $G$ we
have the distributivity "`from right"', that is, by definition
$$(E-F)G=EG-FG,$$
but in general
the "`distributivity from left"' does not hold, that is, in general
$$G(E-F)\not=GE-GF.$$
\item If in Lemma \ref{onecomp} the range $SX$ of $S$ would be
a closed convex cone, then we would be able to assert that $T$ and
$S$ are mutually polar retractions.
\end{enumerate}
\end{remark}

\begin{definition}\label{assnor}\cite{Cobzas2013}
The functional $q:\,X\to \R_+$ is said an \emph{asymmetric norm}
if the following conditions hold:
\begin{enumerate} 
\item $q$ is positive homogeneous, i.e., $q(tx)=tq(x),\, \forall t\in \R_+,\,\,\forall x\in X$;
\item $q$ is subadditive, i.e., $q(x+y)\leq q(x)+q(y),\,\,\forall x,\,y \in X$;
\item If $q(x)=q(-x)=0$ then $x=0$. 
\end{enumerate}
\end{definition}

\begin{example}\label{gauge}
If $C\subset X$ is a closed convex set with $0\in \intr C$,
then the functional $q: \,X\to \R_+$ is called the
\emph{gauge of $C$} if

$$q(x)=\inf\{t\in \R, t>0:\,x\in tC\}.$$

The gauge is an  asymmetric norm (see e. g. \cite{Cobzas2013}, p. 165).
\end{example}

\begin{proposition} \label{??}
If $q:\,X\to \R_+$ is an asymmetric norm
satisfying the relation (\ref{onec}) with some $u$ , $q(u)=1$, then 
$Rx=q(x)u$ and $Q=I-R$ are mutually polar retractions.
\end{proposition}

\begin{proof}
By Lemma \ref{onecomp} and item 2 of Remark \ref{megj} we only have to see
that $\rng Q=Q(X)$ is a closed convex cone.

From $R+Q=I$, it follows that $\rng Q=\ker R =\{x\in X:\,Rx=0\}.$ Hence, we have to show
that $\ker R$ is a closed convex cone. It is closed since $R$ is continuous.
Since $\ker R=\ker q$, we must see that if $q(x)=0,$ then $q(tx)=0$ for $t\in \R_+$
, and that $q(x)=q(y)=0$ implies $q(x+y)=0.$ The first relation follows
from the positive homogeneity of the asymmetric norm. From the
subadditivity of the asymmetric norm we have $0\leq q(x+y)\leq q(x)+q(y)=0.$
This shows that $\rng Q=\ker q$ is a convex cone.

\end{proof}

\begin{proposition}[Proposition 4 \cite{NemethNemeth2020}]\label{cons}

Consider the space $\R\times X$ and denote by $(t,x)$ the sum $t+x,\;t\in \R,\;x\in X$.

Suppose that $g: X\to \R_+$ is an asymmetric norm.

Define
\begin{equation*}
q:\R\times X\to \R_+ \;\;\textrm{by}\;\;q(t,x)=(t+g(x))^+.
\end{equation*}
Then $q$ is an asymmetric norm satisfying the relation
\begin{equation}\label{cons2}
q((t,x)-q(t,x)(1,0))=0,\;\;\forall \;(t,x)\in \R\times X, 
\end{equation}
and hence
\begin{equation}\label{cons3}
Q:\R\times X \to \R\times X,\;\; Q((t,x))=q((t,x))(1,0)
\end{equation}
is a range one subadditive proper retraction.
\end{proposition}

\begin{proof}

The functional $q$ is obviously positively homogeneous.
To prove its subadditivity, we must verify the relation
\begin{equation}\label{cons4}
q((t_1,x_1)+(t_2,x_2))=(t_1+t_2+g(x_1+x_2))^+\leq (t_1+g(x_1))^+ + (t_2+g(x_2))^+= q((t_1,x_1))+ q((t_2,x_2)).
\end{equation}
If all the involved sums in the round brackets  are non-negative, the relation (\ref{cons4}) follows from the subadditivity of $g$.

Suppose that $t_1+g(x_1) \leq 0$. Then
$$t_1+t_2 +g(x_1+x_2) \leq -g(x_1) +t_2 +g(x_1+x_2)\leq -g(x_1) +t_2 +g(x_1) +g(x_2)= t_2 +g(x_2) $$
and the relation (\ref{cons4}) follows independently from the signs of the involved terms.

Suppose $q((t,x))=q(-(t,x))=0$, that is, $t+g(x)=-t +g(-x)=0.$ 
Then $g(x)+g(-x)=0$ and since both the terms are non-negative, they must be zero.
We have first that $x=0$ and then that $t=0$.

Thus $q$ is an asymmetric norm. To verify relation (\ref{onec}), it is sufficient to
consider that $q((t,x))=(t+g(x))^+=1.$ Then we get
$$q((t,x)-(1,0))=0,\;\;\textrm{with} \; q((t,x))=1.$$
But $q((t,x)-(1,0))=q((t-1,x))=(t-1+g(x))^+=0$, and the relation (\ref{onec}) follows.

According to Proposition \ref{??}, the operator $Q$ defined at (\ref{cons3}) is
a proper range one retraction.

\end{proof}

\begin{proposition}\label{cons5}
Let $Q$ be the retraction constructed in Proposition \ref{cons}. Then, by
denoting $K=\R_+(1,0)$, we have for $K^Q=(I-Q)X$ that
$$K^Q=\{(t,x):\,t+g(x)\leq 0\}$$ and
$$(I-Q)((t,x))= (t,x)-q((t,x))(1,0)$$
 is a subadditive proper retraction with
the cone range $K^Q.$
\end{proposition} 

\begin{proof}

Since $Q$ and $I-Q$ are mutually polar retractions, it follows that
$K^Q=\rng(I-Q)=\ker Q=\{(t,x): t+g(x)\leq 0\}.$

Observe that $-K\subset K^Q$ since $-1+g(0)=-1\leq 0,$ that is,
$(-1,0)\in K^Q.$  We have
$$(I-Q)x+(I-Q)y-(I-Q)(x+y)= -Qx-Qy +Q(x+y) \in -K \subset K^Q,$$
which shows the subadditivity of $I-Q$.

\end{proof}

\begin{definition}
The subset $B \subset K$ is called \emph{the basis of the cone} $K$ if for each
$x\in K\setminus \{0\}$ there exists a unique $\lambda >0$ such that
$\lambda x\in B$.
\end{definition}

\begin{lemma}\label{basis}
If the Banach space $X$ is separable, $Q,\,R:X\to X$ are mutually polar
retractions with $M=\rng Q$, $N=\rng R$, and $\dim N=1$, then $M=\rng Q$
poses a basis $B$ on a hyperplane
$H_0$  with $ \spa N \cap \relintr B \not= \emptyset$.
\end{lemma}

\begin{proof}
Since $M$ is a proper cone and $X$ is separable, there exists a hyperplane $H$ through
$0$ such that $(M\setminus \{0\})\cap H=\emptyset$ \cite{KreinRutman1948}. From geometric reasons,
$H$ can be taken such that $H\cap N=\{0\}.$  Since by Proposition \ref{ell} and Proposition \ref{ktrr}
$\spa N\cap \intr M \not=\emptyset$, by taking $u\in (N \setminus \{0\})$, it
follows that $-u \in \intr M$. Then $B= M\cap H_0$ with $H_0=H-u$
will be a basis of $M$ and $-u= \spa N \cap \relintr B \not= \emptyset$.

\end{proof}

The main result of this subsection is the

\begin{theorem}\label{subad} 
Suppose that $X$ is a separable Banach space.
\begin{enumerate}
\item
If $Q,\,R:\,X\to X$ are mutually polar retractions with $\intr\rng M\not= \emptyset$ and
$\dim \rng R=1$,
 then they are well defined by their cone ranges and 
are subadditive.
\item
If $M$ and $N$ are closed transversal cones in $X$ and $\dim N=1$, then there exist
a unique pair of mutually polar retractions $Q$ and $R$ with $\rng Q=M$,
$\rng R=N$. $Q$ and $R$ are subadditive.
\end{enumerate}
\end{theorem}

\begin{proof}

(a) Using the notations in Lemma\ref{basis},
we will identify $H$ with a Banach space.
Translate the basis $B$ along $\spa N$  to the
origin. Denote the translated set by $D$, and let $g$ be
its gauge in $H$.

(b) Let us consider the space $X$ represented as
$$X=\spa N+H.$$
Then $x=tu+y$ with $y\in H$, 
$u\in N\setminus \{0\}$, $t \in \R.$ The element
$x=tu+y$ can be denoted by $x=(t,y)$.

Define $T:X\to X$ as follows:
$$ \textrm{If} \;x=tu+y=(t,y) \;\textrm{put}\; Tx=(t+g(y))^+ u. $$

(c) Repeating step by step the proof in Proposition \ref{cons} 
and Proposition \ref{cons5}, we conclude
that $T$ is a regular range one subadditive retraction on $N$. Thus
$S=I-T$ is a retraction too, with its cone range $\{(t,y):t+g(y)\leq 0\}$.

We will show that
\begin{equation}\label{gauge}
M=\{(t,x): t+g(x)\leq 0\}.
\end{equation}

Take $(t,x)\in M\setminus \{0\}$. Then there exists the unique positive
$\lambda$ with $\lambda (t,x)\in B$. Thus $\lambda (t,x)=(-1,\lambda x).$
Hence, $\lambda t=-1$ and $\lambda x\in D$. Accordingly, $g(\lambda x)\leq 1$.
Putting $\lambda = -\frac{1}{t}$ in the last relation, we get 
$g(-\frac{1}{t} x) \leq 1$ and since $t$ is negative and $g$ 
is positive homogeneous, $g(x)\leq -t$, that is, $t+g(x)\leq 0.$ Hence, 
$(t,x)$ is in the second term of the equality (\ref{gauge}).

Conversely, if $t+g(x)< 0$, then $g(-\frac{1}{t} x)\leq 1$, hence
$-\frac{1}{t} x \in D$. Thus $(-1,-\frac{1}{t} x)\in B$ and
$-t(-1,-\frac{1}{t}x) = (t,x)\in M,$ which completes the 
proof of the relation (\ref{gauge}).

(d) Thus we have constructed the mutually polar, subabbitive
retractions $S$ and $T$ with the cone ranges $M$ and $N$
respectively. But then from Corollary \ref{foo1}, we must
have $Q=S$ and $R=T$. Hence, $Q$ and $R$ are booth 
subadditive, and the proof of item 1 of the theorem is complete.

(e) The proof of item 2 of the theorem follows by using Theorem \ref{foo},
Corollary \ref{foo1} and item 1 of the theorem.

\end{proof}


\subsection{Mutually polar subadditive retractions with
cone ranges with nonempty interiors}

To be in line with our above exposition, we give here without
proof the summary of results in Section 7 in \cite{NemethNemeth2020}:

\begin{theorem}\label{halo}

Suppose that $Q$ and $R$ are mutually polar proper retractions with
the cone ranges $M$ and $N$, respectively. If $\intr M\not=\emptyset$
and $\intr N \not=\emptyset$, then the following
conditions are equivalent:

\begin{enumerate}
\item $Q$ and $R$ are subadditive;
\item $Q$ and $R$ are isotone;
\item $M$ or $N$ are lattice cones,
$N=-M$, and $Q$, $R=I-Q$ are the positive part operators
in the vector lattices endowed by the positive cones $M$ and $N$, respectively.
\end{enumerate}
If $\dim H< \infty$ then besides the above equivalent properties we have
via the Theorem of Youdine the equivalent condition:
4. The cone range of $Q$ is a simplicial cone $K$ and the cone range
of $R$ is the simplicial cone $-K$.

\end{theorem}


\subsection{Metric projection on the cone with order theoretic
properties}

Some results in Section 8 of \cite{NemethNemeth2020} can be
gathered as follows:

\begin{theorem}\label{mprj}

If $K$ is a nonempty generating closed cone in the separable Hilbert space,
then for the metric projection $P$ onto $K$ and its polar $I-P$ we have:
\begin{enumerate}
\item If $P$ is subadditive, then $I-P$ is isotone;
\item If $P$ is isotone, then $I-P$ is subadditive;
\item $P$ and $I-P$ are subadditive (isotone) if and
only if $(H,K)$ is a Hilbert vector lattice;
\item If $\dim H<\infty$ the last condition is equivalent
with the condition that $K$ is the positive
orthant of a Cartesian reference system.
\end{enumerate}

\end{theorem}


\end{document}